\documentclass[a4paper,12pt]{amsart}

\usepackage[english]{babel}
\usepackage{amsmath,amssymb}
\usepackage[alphabetic]{amsrefs}
\usepackage[unicode,pdfencoding=auto,psdextra]{hyperref}
\usepackage{etoolbox}
\apptocmd{\sloppy}{\hbadness 10000\relax}{}{}


\newtheorem{theorem}{Theorem}
\newtheorem{corollary}[theorem]{Corollary}

\newtheorem{lemma}[theorem]{Lemma}
\newtheorem{proposition}[theorem]{Proposition}
\theoremstyle{definition}
\newtheorem{definition}[theorem]{Definition}
\newtheorem{example}[theorem]{Example}
\theoremstyle{remark}
\newtheorem{remark}[theorem]{Remark}

\newcommand{\CC}{\mathbb{C}}

\newcommand{\NN}{\mathbb{N}}
\newcommand{\camo}{\mathcal{C}}
\newcommand{\cont}{\mathcal{C}}
\newcommand{\holo}{\mathcal{O}}
\newcommand{\matnn}{\mathrm{Mat}(n \times n; \CC)}
\DeclareMathOperator{\tr}{tr}
\DeclareMathOperator{\rank}{rank}
\DeclareMathOperator{\id}{id}
\DeclareMathOperator{\gl}{GL}
\DeclareMathOperator{\lie}{Lie}

\title[Symplectic density property for Calogero--Moser spaces]{The symplectic holomorphic density property for Calogero--Moser spaces}
\author{Rafael B. Andrist}
\address{Department of Mathematics \\
American University of Beirut \\
Beirut, Lebanon \\ And Faculty of Mathematics and Physics \\
University of Ljubljana \\
Ljubljana, Slovenia}
\email{rafael-benedikt.andrist@fmf.uni-lj.si}

\author{Gaofeng Huang}
\address{University of Bern \\ Mathematical Institute \\ 
Bern, Switzerland}
\email{gaofeng.huang@unibe.ch}

\begin{document}

\begin{abstract}
We introduce the symplectic holomorphic density pro\-perty and the Hamiltonian holomorphic density property together with the corresponding version of Anders{\'e}n--Lempert theory. We establish these properties for the Calogero--Moser space $\mathcal{C}_n$ of $n$ particles and describe its group of holomorphic symplectic automorphisms. 
\end{abstract}



\maketitle

\section{Introduction}

The modern study of large holomorphic automorphism groups of Stein manifolds started with the seminal paper of Rosay and Rudin \cite{MR929658} in 1988. With the works of And{\'e}rsen and Lempert \cite{MR1185588} in 1992 and Forstneri\v{c} and Rosay \cite{MR1213106} in 1993, a description of and an approximation theory for the holomorphic automorphism group of $\CC^n$ was obtained. This was generalized by Varolin \cites{MR1785520, MR1829353} to Stein manifolds with the density property which he introduced around 2000. This area of research, now called \emph{Anders{\'e}n--Lempert theory}, has developed rapidly, and many Stein manifolds with the density property have been found since then; We refer to the very recent survey by Forstneri\v{c} and Kutzschebauch \cite{MR4440754} for more details.

However, only very little is known so far about symplectic holomorphic automorphisms of Stein manifolds that are equipped with a symplectic holomorphic form. For the cotangent bundle $T^\ast \CC^n \cong \CC^{2n}$ of the complex-Euclidean space with the standard symplectic form, the group of symplectic holomorphic automorphisms was described by Forstneri\v{c} \cite{MR1408866} in 1996. In a recent preprint by Berger and Turaev \cite{berger-turaev}, Forstneri\v{c}'s result has also been transferred to the real smooth symplectic case. Recently, Deng and Wold \cite{MR4423269} deduced from Forstneri\v{c}'s result the Hamiltonian density property for closed co\-adjoint orbits of complex Lie groups. However, the only explicit examples they gave are surfaces where the symplectic form equals the holomorphic volume form and hence these examples are already covered by the well-known volume density property. No other descriptions of the full symplectic holomorphic automorphism group for Stein manifolds have been obtained. %In particular, we do not know how to describe the group of symplectic holomorpic automorphisms of the cotangent bundle of a complex Lie group. 

In this paper, we introduce the \emph{symplectic holomorphic density property} and the \emph{Hamiltonian holomorphic density property} for Stein manifolds together with an appropriate version of the Anders\'en--Lempert theory in Section \ref{sec-dp} which has been mentioned for $\mathbb{C}^{2n}$ in the survey article \cite{MR4440754}. As a first example, we then prove the Hamiltonian/symplectic density property for $T^\ast \CC^n$  in Section \ref{sec-tcn} and derive the result of Forstneri\v{c} \cite{MR1408866}*{Theorem 5.1} about its group of symplectic holomorphic automorphisms; In fact, we can choose a slightly smaller number of generators, see Theorem \ref{thm-cn} and its corollary.

Then, we turn to the main subject of this paper, the \emph{Calogero--Moser spaces} $\camo_n$ which are equipped with a symplectic holomorphic form. In Section \ref{sec-CaMo} we discuss some of their properties and their topology; it turns out \cite{MR1626461} that $\camo_n$ is diffeomorphic to the Hilbert scheme of $n$ points in the affine plane $\CC^2$ whose topology has been fully described by Ellingsrud and Str{\o}mme \cites{MR870732, MR922805}. Since its first de Rham cohomology group vanishes, every symplectic holomorphic vector field on $\camo_n$ is Hamiltonian, see our Lemma \ref{lem-camotopo}.

Calogero--Moser spaces have been a very interesting object of study in pure mathematics for a few decades. We follow the presentation given by Wilson \cite{MR1626461}. These spaces arise as a completion of the phase space of $n$ classical, indistinguishable, point-like particles moving in $\CC$ according to a Hamiltonian with a quadratic inverse potential. This prohibits any two particles from having the same spatial coordinate. The Calogero--Moser space $\camo_n$ is a completion of this phase space, allowing the particles to ``collide''. 

\begin{definition}
Let $\widetilde{\camo_n}$ be the subvariety of $\matnn \times \matnn$ given by 
\[
\rank([X,Y] + \id) = 1
\]
where $(X,Y) \in \matnn \times \matnn$. The group $\gl_n(\CC)$ acts on $\widetilde{\camo_n}$ by simultaneous conjugation in both factors:
\[
g \bullet (X,Y) = (g \cdot X \cdot g^{-1}, g \cdot Y \cdot g^{-1})
\]
for $g \in \gl_n(\CC)$.
We define the \emph{Calogero--Moser space} $\camo_n$ of $n$ particles to be the GIT-quotient $\widetilde{\camo_n}/\!/\gl_n(\CC)$.
\end{definition}

It is easy to see that the group action is well defined, since it leaves the condition 
$\rank([X,Y] + \id) = 1$ invariant. 

For $n = 1$ the matrices are just complex numbers, and we obtain $\camo_1 = \CC^2$.
In higher dimensions, the structure becomes more sophisticated. The Calogero--Moser space $\camo_n$ is a smooth complex-affine variety of dimension $2n$, see Wilson \cite{MR1626461}*{Section 1}. More recently, it has been shown by Popov \cite{MR3200435}*{Section 13, Remark 5} that $\camo_n$ is a rational variety.

The Calogero--Moser space is equipped with a symplectic holomorphic form and in fact carries a hyperkähler structure. Indeed, the Calogero--Moser space with its symplectic form is obtained as symplectic reduction of the preimage of a coadjoint orbit under a moment map, see Section \ref{sec-CaMo} for more details.

We will make use of the following well-known complete flow maps.
\begin{definition}
The \emph{Calogero--Moser flows} on $\camo_n$ are defined as follows:
\begin{align*}
(X,Y) \mapsto (X, Y + t X^k) \\
(X,Y) \mapsto (X + t Y^k, Y)
\end{align*}
where $k \in \NN_0$ and $t \in \CC$.
\end{definition}

It is easy to see that these are well-defined maps on $\camo_n$ since they are invariant under conjugation and leave the commutators $[X,Y]$ invariant, and that they are algebraic isomorphisms for each $t \in \CC$. Moreover, they leave the symplectic form invariant, see Section \ref{sec-CaMo}. 

It has been established by Berest and Wilson \cite{MR1785579} that the automorphism group acts transitively on the Calogero--Moser space. This was later improved to $2$-transitivity by Berest, Eshmatov and Eshmatov \cite{MR3459703} and finally to $m$-transitivity for any $m \in \NN$ (sometimes called \emph{infinite transitivity}) by Kuyumzhiyan \cite{MR4127820}. In fact, all these three results are slightly stronger, since they do not make use of the full group of automorphisms, but only of the subgroup generated by the Calogero--Moser flows which are in fact symplectic holomorphic automorphisms.

The density property for the Calogero--Moser spaces has been established by the first author \cite{MR4305975}. As a consequence, every holomorphic automorphism that is homotopic to the identity can be approximated by shears and overshears of Calogero--Moser flows. It also follows from that result that the group of holomorphic automorphisms acts $m$-transitively for any $m \in \NN$.

In this paper, we prove the analogous result for symplectic holomorphic automorphisms which turns out to be more involved. For the definition of Hamiltonian/symplectic holomorphic density property, see Section \ref{sec-dp}.

\begin{theorem}
\label{thm-hamiltonDP}
The Calogero--Moser space $\camo_n$, $n \in \NN$, has the Hamiltonian holomorphic density property.
\end{theorem}

In fact, we will show that the Hamiltonian holomorphic density property can be established using only a finite number of complete Hamiltonian holomorphic vector fields that correspond to the four Hamiltonian functions $\tr Y, \tr Y^2, \tr X^3, (\tr X)^2$. Such properties have been established before by the first author for the case of $\CC^n$ \cite{MR3975669} and for $\mathrm{SL}_2(\CC)$ as well as $x y = z^2$ \cite{finite-lie-bis} with finitely many complete polynomial vector fields generating the Lie algebra of all polynomial vector fields.

\begin{corollary}
The Calogero--Moser space $\camo_n$, $n \in \NN$, has the symplectic holomorphic density property.
\end{corollary}
\begin{proof}
This follows from the preceding theorem and Lemma \ref{lem-camotopo}.
\end{proof}

\begin{theorem}
\label{thm-camo-autos}
The identity component of the group of symplectic holomorphic automorphisms of $\camo_n$, $n \in \NN$, is the closure (in the topology of uniform convergence on compacts) of the group generated by the following symplectic holomorphic automorphisms:
\begin{align*}
(X,Y) &\mapsto (X + t \, \mathrm{id}, Y ) \\
(X,Y) &\mapsto (X + t \, Y, Y) \\
(X,Y) &\mapsto (X, Y + t \, X^2) \\
(X,Y) &\mapsto (X, Y + t \tr(X) \cdot \id ) 
\end{align*}
where $t \in \CC$.
\end{theorem}

We currently do not know whether the symplectic holomorphic automorphism group has more than one connected component.

\begin{corollary}
The group generated by the symplectic holomorphic automorphisms in the preceding theorem acts $m$-transitively on $\camo_n$ for any $m \in \NN$.
\end{corollary}

\begin{proof}
This follows from the symplectic holomorphic density property and Corollary \ref{cor-trans}.
\end{proof}

%In the process of establishing the Hamiltonian holomorphic density property, we also obtain a description of the invariant functions on the subvariety given by $\rank([X,Y] + \id) = 1$ where $(X,Y) \in \matnn \times \matnn$, i.e.\ the space denoted by $\widetilde{\camo_n}$.
%
%\begin{theorem}
%The ring of (algebraic) invariant functions on $\widetilde{\camo_n}$ is generated by
%\[
%\tr(X^p Y^q)
%\]
%where $0 \leq p + q \leq   n^2 $.
%\end{theorem}
%
%In fact, we prove that on $\widetilde{\camo_n}$, the matrices inside the trace ``essentially'' commute, and that the difference of two traces of monomials in $X$ and $Y$ with the same degrees in $X$ and $Y$ only differ by traces of polynomials of total degree that is smaller by $4$.
%
%\begin{example}
%On $\widetilde{\camo_n}$ we have that
%\[
%\tr(X^2 Y^2) - \tr(X Y X Y) = \binom{n}{2}
%\]
%according to Lemma \ref{lemma: trXYB}.
%\end{example}
%
%Our description of the invariant functions is based on the results of Razmyslov \cite{MR0506414} and Procesi \cite{MR419491} who gave generators and relations for the ring of invariant functions on $(\matnn)^m$. For $n \geq 4$ it seems to become too technical to handle the relations between their generators for any practical calculations. However, our result for $\widetilde{\camo_n}$ uses fewer generators and hence considerably simplifies the situation. We believe it to be of independent interest.

\section{Hamiltonian and Symplectic Holomorphic Density Property}
\label{sec-dp}

The density property and the volume density property for complex manifolds were introduced by Varolin \cites{MR1829353, MR1785520}. The Hamiltonian density property was also studied by Deng and Wold \cite{MR4423269}*{Section 3.1} in a recent article; they proved an Anders\'en--Lempert theorem with Carleman approximation. The symplectic Anders\'en--Lempert theorem was stated in \cite{MR4440754} for $\mathbb{C}^{2n}$. %but the theorem without Carleman approximation was not stated by them in general setting.

\begin{definition}
Let $X$ be a complex manifold with a symplectic holomorphic form $\omega$.
\begin{enumerate}
\item We call a holomorphic vector field $V$ on $X$ \emph{symplectic} if $\mathcal{L}_V \omega = 0$.
\item We call a holomorphic vector field $V$ on $X$ \emph{Hamiltonian} if there exists a holomorphic function $H \colon X \to \CC$, called the \emph{Hamiltonian function of $V$}, such that $i_V \omega = d H$. 
\end{enumerate}
\end{definition}

\begin{remark}
Since $\omega$ is closed, it follows from Cartan's homotopy formula that $V$ is symplectic if and only if $d i_V \omega = 0$. Further, this implies that every Hamiltonian vector field is symplectic.
%From the following formula \cite{MR2723362}*{Problem 20.8, p.~234}
%\[
%i_{[V,W]} \omega = [\mathcal{L}_V, i_W]\omega 
%\]
%and Cartan's homotopy formula again an explicit formula for the Hamiltonian function %of the Lie bracket of two Hamiltonian vector fields $V$ and $W$
%can be derived, see \cite{MR3492044}*{Proposition 3.1}:
%\[
%i_{[V,W]} \omega = d i_V i_W \omega
%\]
The symplectic holomorphic vector fields and the Hamiltonian holomorphic vector fields on a complex manifold each form a Lie algebra.
\end{remark}

The holomorphic functions on $X$ form a Lie algebra under the \emph{Poisson bracket} $\{ \cdot , \cdot \}$. It is defined in such a way that the correspondence between Hamiltonian functions $H$ and $K$ and their respective Hamiltonian vector fields $V$ and $W$ respects the Lie algebra structure:
\[
i_{[V,W]} \omega = -d \{ H, K \}
\]
see the textbook of Arnol'd \cite{MR1345386}*{Section 40 Corollary 3, p.~215}.

\begin{definition}
Let $X$ be a complex manifold with a symplectic holomorphic form $\omega$.
\begin{enumerate}
\item We say that $(X, \omega)$ has the \emph{symplectic holomorphic density property} if the Lie algebra generated by the complete symplectic holomorphic vector fields on $X$ is dense (in the topology of locally uniform convergence) in the Lie algebra of all symplectic holomorphic vector fields on $X$.
\item We say that $(X, \omega)$ has the \emph{Hamiltonian holomorphic density property} if the Lie algebra generated by the complete Hamiltonian holomorphic vector fields on $X$ is dense (in the topology of locally uniform convergence) in the Lie algebra of all Hamiltonian holomorphic vector fields on $X$.
\end{enumerate}
\end{definition}

\begin{remark}
We consider the space of functions $\holo(X) = \{ f \colon X \to \CC \; \text{holomorphic} \}$ as a Lie algebra with the Poisson bracket $\{\cdot,\cdot\}$. 
Note that the Hamiltonian density property is equivalent to stating that this Lie algebra contains a dense (in the topology of locally uniform convergence) Lie sub-algebra that is generated by those functions that correspond to complete Hamiltonian vector fields. To see this, we only need to observe that for holomorphic functions the locally uniform convergence implies also the locally uniform convergence of the derivatives due to Cauchy estimates.

We also emphasize that we only work with the Lie algebra structure, and not with the Poisson algebra structure, i.e.\ we are not allowed to multiply functions. The conclusions of the density property only hold when working with Lie combinations of the corresponding vector fields.
\end{remark}

In the holomorphic case, i.e.\ without extra conditions about being symplectic or Hamiltonian and in the volume preserving holomorphic case, the following theorem is due to Anders{\'e}n and Lempert \cite{MR1185588}, Forstneri\v{c} and Rosay \cites{MR1213106, MR1296357} and Varolin \cite{MR1829353}.

\begin{theorem}
\label{thm-AL}
Let $X$ be a Stein manifold with a symplectic holomorphic form $\omega$. Assume that $(X, \omega)$ has the symplectic holomorphic density property or the Hamiltonian holomorphic density property.

Let $\Omega \subset X$ be an open Stein subset with holomorphic de Rham cohomology $H^1(\Omega, \CC) = 0$ and let $\varphi_t \colon \Omega \to X, t \in [0,1],$ be a jointly $\cont^1$-smooth map such that the following holds:
\begin{enumerate}
\item $\varphi_0 \colon \Omega \to X$ is the natural embedding.
\item $\varphi_t \colon \Omega \to X$ is a symplectic holomorphic injection for each $t \in [0,1]$.
\item $\varphi_t(\Omega)$ is a Runge subset of $X$ for each $t \in [0,1]$.
\end{enumerate}
Then for every compact $K \subset \Omega$ and every $\varepsilon > 0$ and every choice of metric on $X$ that induces its topology, there exists
a continuous family $\Phi_t \colon X \to X$ of symplectic holomorphic automorphisms  such that
\[
\sup_{x \in K} d(\varphi_1(x), \Phi_1(x)) < \varepsilon
\]
Moreover, $\Phi_t$ can be written as a finite composition of flows of complete vector fields that are generators of the Lie algebra of symplectic resp. Hamiltonian holomorphic vector fields on $X$.
\end{theorem}

\begin{proof}
The proof is basically the same as for the case of volume density property which in turn is a modification of the proof for the density property, see Kaliman and Kutzschebauch \cite{MR2768636}*{Remark 2.2}; We only need to replace the $n$-form by the closed $2$-form $\omega$.
Since we require the first holomorphic de Rham cohomology to vanish, we can assume that any symplectic vector field on $\Omega$ is Hamiltonian. This is analogous to the case of the volume density property where the $n$-form induces a correspondence between vector fields and $(n-1)$-forms through $i_V \omega = d \eta$. In that case, $n \geq 2$ and $H^{n-1}(X, \CC) = 0$ is required in order to write every closed $(n-1)$-form $\eta$ as the exterior derivative of an $(n-2)$-form $\zeta$, and then Runge approximation is used for $\zeta$. In our case, this is simply the correspondence between the Hamiltonian function and its Hamiltonian vector field, and we can directly use Runge approximation for the Hamiltonian function. 
\end{proof}

The following corollary then follows again the same way as for the volume density property, see Kaliman and Kutzschebauch \cite{MR2768636}*{Remark 2.2}.

\begin{corollary}
\label{cor-trans}
Let $X$ be a Stein manifold with a symplectic holomorphic form $\omega$. Assume that $(X, \omega)$ has the symplectic holomorphic density property or the Hamiltonian holomorphic density property. Then the group of symplectic holomorphic automorphisms acts $m$-transitively on $X$ for any $m \in \mathbb{N}$. 
\end{corollary}

\section{Warm-Up: $T^\ast \CC^n$}
\label{sec-tcn}

The group of symplectic holomorphic automorphisms of $T^\ast \CC^n \cong \CC^{2n}$ with the standard symplectic holomorphic form has been described already by Forstneri\v{c} \cite{MR1408866}. However, that proof did not make full use of the Lie algebra structure. Here, we first prove the symplectic holomorphic density property for $T^\ast \CC^n \cong \CC^{2n}$ and then obtain the description of the symplectic holomorphic automorphism group as direct consequence.

Let $(x, y) \in \CC^n \times \CC^n \cong T^\ast \CC^n$ and let $\omega = \sum_{k=1}^{n} dx_k \wedge dy_k$ be the canonical symplectic form on $T^\ast \CC^n$. The Poisson bracket is then given by $\{ x_j, y_k \} = \delta_{jk}$, $\{ x_j, x_k \} = 0$ and $\{ y_j, y_k \} = 0$.

We first consider the following Hamiltonian functions and the corresponding Hamiltonian vector fields that are complete:
\begin{align}
&H = \frac{-1}{p+1} x_k^{p+1} &\Longrightarrow V = x_k^p \frac{\partial}{\partial y_k} \label{eq-Cn-Hamilton1} \\
&H = \frac{1}{q+1} y_k^{q+1} &\Longrightarrow V = y_k^q \frac{\partial}{\partial x_k} \label{eq-Cn-Hamilton2} \\
&H = -x_j x_k &\Longrightarrow V = x_j \frac{\partial}{\partial y_k} + x_k \frac{\partial}{\partial y_j} \label{eq-Cn-Hamilton3} \\
&H = y_j y_k &\Longrightarrow V = y_j \frac{\partial}{\partial x_k} + y_k \frac{\partial}{\partial x_j} \label{eq-Cn-Hamilton4}
\end{align}

\begin{theorem}
\label{thm-cn}
$T^\ast \CC^n$ has the Hamiltonian/symplectic density property. A dense Lie-subalgebra is generated by the vector fields from Equations \eqref{eq-Cn-Hamilton1}, \eqref{eq-Cn-Hamilton2}, \eqref{eq-Cn-Hamilton3} and \eqref{eq-Cn-Hamilton4}.
\end{theorem}

\begin{proof} It suffices to find all monomial Hamiltonian functions, as we can then take linear combinations and limits.
\begin{enumerate}
\item The only missing terms of pure degree two are obtained by
$\{ x_k^2, y_k^2 \} = 4 x_k y_k$ and $\{ x_k^2, y_j y_k \} = 2 x_k y_j$.
\item For each index $k = 1, \dots, n$ we obtain
\[
\{ x_k^p, y_k^q \} = 2pq x_k^{p-1} y_k^{q-1}
\]
\item We proceed by induction in $k=1, \dots, n$. If certain variables do not appear at all, we can renumber the indices accordingly. Assume that we have already all monomials in the variables $x_1, y_1, \dots, x_k, y_k$. The base case $k=1$ is obvious. For the induction step from $k$ to $k+1$ we start with powers $p(k+1) = 0$ and $q(k+1) = 0$, and proceed again by induction first in $p(k+1)$ and then in $q(k+1)$:
\begin{align*}
&\{ x_1^{p(1)} y_1^{q(1)} \cdots x_k^{p(k)} y_k^{q(k)} \cdot x_{k+1}^{p(k+1)}, x_k x_{k+1} \} \\
&= - q(k) \cdot x_1^{p(1)} y_1^{q(1)} \cdots x_k^{p(k)} y_k^{q(k)-1} x_{k+1}^{p(k+1)+1} \\
&\{ x_1^{p(1)} y_1^{q(1)} \cdots x_k^{p(k)} y_k^{q(k)} \cdot x_{k+1}^{p(k+1)} \cdot y_{k+1}^{q(k+1)}, x_k y_{k+1} \} \\ 
&= - q(k) \cdot x_1^{p(1)} y_1^{q(1)} \cdots x_k^{p(k)} y_k^{q(k)-1} x_{k+1}^{p(k+1)} \cdot y_{k+1}^{q(k+1)+1} \\
&+ p(k+1) \cdot x_1^{p(1)} y_1^{q(1)} \cdots x_k^{p(k)+1} y_k^{q(k)} x_{k+1}^{p(k+1)-1} \cdot y_{k+1}^{q(k+1)} \qedhere
\end{align*}
\end{enumerate}
\end{proof}

\begin{corollary}[\cite{MR1408866}*{Theorem 5.1}]
Every symplectic holomorphic automorphism of $T^\ast \CC^n$ can be approximated uniformly on compacts by compositions of symplectic shears.
\end{corollary}
\begin{proof}
We only need to notice that the flows of the Hamiltonian vector fields in Equations \eqref{eq-Cn-Hamilton1}, \eqref{eq-Cn-Hamilton2}, \eqref{eq-Cn-Hamilton3} and \eqref{eq-Cn-Hamilton4} are symplectic shears and then apply the Anders{\'e}n--Lempert theorem.
\end{proof}

\section{The Calogero--Moser space}
\label{sec-CaMo}

We follow the definition given by Wilson \cite{MR1626461}.

\begin{lemma}\cite{MR1626461}*{Proposition 1.10}
Let $(X,Y) \in \matnn \times \matnn$ such that $\rank([X, Y] + \id) = 1$. If $X$ is diagonalizable, then all eigenvalues of $X$ are pairwise different, and there exists $g \in \gl_n(\CC)$ such that 
\begin{equation}
\label{coordWilson}
\begin{split}
&(g X g^{-1}, g Y g^{-1}) = \\
&\left( \begin{pmatrix}
\alpha_1 \\
& \alpha_2 \\
& & \ddots \\
& & & \alpha_n 
\end{pmatrix} \right., \\
&\left. \begin{pmatrix}
\beta_1 & (\alpha_1 - \alpha_2)^{-1} & \dots & (\alpha_1 - \alpha_n)^{-1} \\
(\alpha_2 - \alpha_1)^{-1} & \beta_2 & \ddots & \vdots \\
\vdots & \ddots & \ddots & (\alpha_{n-1} - \alpha_{n})^{-1} \\
(\alpha_n - \alpha_1)^{-1} & \dots & (\alpha_{n} - \alpha_{n-1})^{-1} & \beta_n
\end{pmatrix}
  \right)
\end{split}
\end{equation}
for $\alpha_1, \dots, \alpha_n, \beta_1, \dots, \beta_n \in \CC$ with $\alpha_j \neq \alpha_k$ for $j \neq k$. 
Moreover, $(g([X, Y] + \id) g^{-1})_{jk} = 1$ for all $j,k = 1, \dots, n$.
\end{lemma}

Note that the order of the eigenvalues $\alpha_1, \dots, \alpha_n$ is arbitrary, hence we obtain an $n!$-to-$1$ covering of the open and dense subset of $\camo_n$ where the matrices $X$ are diagonalizable. Alternatively, we can define the injective mapping
\[
\Phi \colon \left( \{ (\alpha_1, \dots \alpha_n) \in \CC^n \, : \, \alpha_j \neq \alpha_k \text{ for } j \neq k \} \times \CC^n \right) / S_n \to \camo_n
\]
that maps $((\alpha_1, \dots, \alpha_n), (\beta_1, \dots, \beta_n))$ to $(X,Y)$ according to to Equation \eqref{coordWilson}. Here, $S_n$ denotes the symmetric group that acts by simultaneous permutations on $(\alpha_1, \dots, \alpha_n)$ and $(\beta_1, \dots, \beta_n)$. We will refer to $\Phi$ as the \emph{Wilson coordinates}.

For the definition of vector- or matrix-valued differential forms, see e.g.\ the textbook of Tu \cite{MR2723362}*{Section 21}. The standard conjugation-invariant symplectic form on $\matnn \times \matnn$ according to Wilson \cite{MR1626461}*{p.~9} is given as
\begin{equation}
\widetilde{\omega} = \tr (dX \wedge dY) = \sum_{j, k=1}^n dX_{jk} \wedge dY_{kj}
\end{equation}
The form $\widetilde{\omega}$ is invariant under conjugation. Following Etingof \cite{MR2296754}, the action of $\gl_n(\CC)$ admits a moment map %it descends to the quotient where it can be restricted to the submanifold $\camo_n$. We denote this restriction by $\omega$. It is clear that $\omega$ is also a closed $2$-form.
\begin{align*}
	\mu \colon \matnn \times \matnn \to \mathfrak{gl}_n (\mathbb{C}), \; (X,Y) \mapsto [X, Y]
\end{align*}
where we identify the Lie algebra with its dual using the trace form $\left< M, N\right> =  \tr(MN)$. This moment map was first given in \cite{MR1626461}, following the construction with a unitary group action in Kazhdan, Kostant and Sternberg \cite{MR478225}.
%Let $\Phi$ be the Wilson coordinates from Equation \eqref{coordWilson}. Then a straightforward computation shows $\Phi^\ast(\omega) = \sum_{j=1}^n d \alpha_j \wedge d \beta_j$, hence $\omega$ coincides with the standard symplectic form in these coordinates, in particular it is non-degenerate. To see that $\omega$ is non-degenerate everywhere on $\camo_n$, we can push forward $\omega$ by the Calogero--Moser flows which preserve $\omega$. Since they generate a group that acts transitively \cite{MR1785579}*{Theorem 1.3}, $\omega$ is non-degenerate on $\camo_n$. 
%Another way to see this is the construction of the Calogero--Moser space as the symplectic reduction of $\matnn \times \matnn$ by the group action of $\gl_n(\CC)$ under simultaneous conjugation, as carried out in the paper by Kazhdan, Kostant and Sternberg \cite{MR478225}. 

Let $O_{\xi}$ be the coadjoint orbit of the matrix 
$$\xi = \mathrm{diag} (-1, -1, \dots, -1, n-1)$$ 
Since the coadjoint action of $\gl_n(\CC)$ on $\mathfrak{gl}^*_c(\CC)$ is given by conjugation, $O_{\xi}$ consists of traceless matrices $T$ such that $T+\mathrm{id}$ is of rank one.
Then $\widetilde{\mathcal{C}_n}= \mu^{-1} (O_{\xi})$ is the preimage of this orbit, upon which $\gl_n(\CC)$ acts freely. The symplectic reduction along this orbit 
\begin{align*}
	\pi \colon \mu^{-1} (O_{\xi}) \to  \mu^{-1} (O_{\xi}) / \mathrm{GL}_n (\mathbb{C}) = \mathcal{C}_n
\end{align*}
gives the symplectic form $\omega$ on $ \mathcal{C}_n$ which satisfies $ \pi^* \omega = i^* \widetilde{\omega} $, where $i \colon \mu^{-1} (O_{\xi}) \hookrightarrow \matnn \times \matnn$ denotes the inclusion.

\bigskip

The Calogero--Moser space $\camo_n$ is diffeomorphic to the Hilbert scheme of $n$ points in the affine plane, $\mathrm{Hilb}_n(\CC^2)$. This can be shown using the existence of a hyperkähler structure on $\camo_n$, see Wilson \cite{MR1626461}*{Section 8}. For $n>1$ the Hilbert scheme is not affine or Stein, hence this diffeomorphism cannot be an isomorphism in the algebraic or holomorphic category. The topology of the Hilbert scheme of a plane is well-known; The Borel--Moore homology of $\mathrm{Hilb}_n(\CC^2)$ has been calculated by Ellingsrud and Str{\o}mme \cite{MR870732}*{Theorem 1.1, (iii)}. In particular, they obtain that the odd homology vanishes. This homology had been introduced by Borel and Moore \cite{MR131271} to obtain Poincar{\'e} duality for singular cohomology on non-compact manifolds. Since the real dimension of a complex manifold is even, it follows that all odd cohomology vanishes. This implies by the universal coefficient theorem that $H^1(\mathrm{Hilb}_n(\CC^2), \CC) = 0$ which is a topological invariant, hence $H^1(\camo_n, \CC) = 0$. Since $\camo_n$ is an affine manifold, this implies that both the algebraic and the holomorphic first de Rham cohomology group is trivial, since they can be computed using resolutions of the sheaf of locally constant complex-valued functions. We conclude the following:
\begin{lemma}
\label{lem-camotopo}
We have that $H^1(\camo_n, \CC) = 0$. Hence all symplectic holomorphic vector fields on $\camo_n$ are in fact Hamiltonian holomorphic vector fields.
\end{lemma}

The last three sections will contain the main part of the proof for the Hamiltonian holomorphic density property. Instead of dealing with the vector fields directly, we will consider the corresponding Hamilton functions. The following basic remark is crucial for our calculations. It is a consequence of the construction of the Calogero--Moser space by symplectic reduction.% we give an elementary explanation.

\begin{remark}
Let $f, h$ be two Hamiltonian functions on the Calogero--Moser space $\cont_n$ and let $F$ and $H$ be $\gl_n(\CC)$-invariant extensions of $f \circ \pi$ and $h \circ \pi$, respectively. Then 
\begin{align*}
	\{ F, H \}_{\widetilde{\omega}} \circ i = \{ f, h \}_{\omega} \circ \pi
\end{align*}
and their corresponding vector fields and flows are related in a similar manner. In this way the symplectic reduction relates the Poisson structure on $\mathcal{C}_n$ to the Poisson structure on $\matnn \times \matnn$, see Marsden and Ratiu \cite{MR836071} for an explanation in the context of Poisson reduction. In other words, to obtain Poisson brackets between Hamiltonian functions on $\mathcal{C}_n$, it suffices to compute brackets of the corresponding invariant Hamiltonian functions on $\matnn \times \matnn$.
\end{remark}

Given two Hamiltonian functions $F$ and $H$ on $\matnn \times \matnn$, their Poisson bracket associated with the symplectic form $\widetilde{\omega} = \sum_{j,k} dX_{jk} \wedge dY_{kj}$ is
\begin{align*}
    \{ F, H \} 
    =  \sum_{j,k = 1}^n \frac{\partial F}{\partial X_{jk}} \frac{\partial H}{\partial Y_{kj}} - \frac{\partial F}{\partial Y_{jk}} \frac{\partial H}{\partial X_{kj}}  %= \tr (\nabla_X f \cdot \nabla_Y g - \nabla_Y f \cdot \nabla_X g ).
\end{align*}
We recall that the Poisson bracket is antisymmetric and satisfies the Leibniz rule, namely
\begin{align*}
    \{ F, H \} = - \{ H, F \}, \quad  \{F_1 F_2, H \} = F_1 \{F_2, H \} + F_2 \{F_1, H \}
\end{align*}
By the Leibniz rule
\begin{align*}
    \{ F^j, H^k \} = j k F^{j-1} H^{k-1} \{F, H \}, \quad j, k \ge 1. 
\end{align*}

\begin{lemma} \label{lemma:completeVFs}
    Hamiltonian functions of the form $H(X)$ or $H(Y)$, depending either only on $X$ or only on $Y$, and $H = \tr XY$ induce complete vector fields.
\end{lemma}
\begin{proof}
    By definition, given a Hamiltonian function $H(X,Y)$, there is a unique vector field $V_H$ satisfying:
    \begin{align*}
        i_{V_H} \widetilde{\omega} = \mathrm{d} H.
    \end{align*}
 The Hamiltonian vector field $V_H$ is 
    \begin{align*}
        V_H = \sum_{j,k} \frac{\partial H}{\partial Y_{jk}} \frac{\partial }{\partial X_{kj}} - \sum_{j,k} \frac{\partial H}{\partial X_{jk}} \frac{\partial }{\partial Y_{kj}} 
        %= \tr \left( \nabla_Y H \cdot \frac{\partial }{\partial X} - \nabla_X H \cdot \frac{\partial }{\partial Y} \right).
    \end{align*}
    Now if $H = H(X)$ only depends on $X$, then the first summand vanishes. We obtain
    \begin{align*}
        V_H = - \sum_{j,k} \frac{\partial H}{\partial X_{jk}} \frac{\partial }{\partial Y_{kj}} 
    \end{align*}
    and the coefficients depend only on $X$, which are constant along any trajectory of a local flow. Hence $V_H$ is complete. For $H= \tr XY$ the associated vector field is
    \begin{align*}%\label{equation: VFtrXY}
        V_H = \sum_{j,k} X_{kj}\frac{\partial}{\partial X_{kj}} - \sum_{j,k} Y_{kj}\frac{\partial}{\partial Y_{kj}}
    \end{align*}
    which has coefficients linear in each variable and is complete. 
\end{proof}

\begin{example} \label{example: CMHamVectFlow}
We give a few examples of Hamiltonian functions, their corresponding vector fields and flows. In particular, the examples given here are complete, i.e.\ their flows exist for all complex times.

\medskip

\begin{tabular}{c|c|c}
Hamiltonian & vector field & flow map $((X,Y),t) \mapsto$ \\
\hline
$\tr X^j$ & $ - j X^{j-1} \frac{\partial}{\partial Y}$ & $(X, Y - t j X^{j-1})$ \\
$\tr Y^j$ & $j Y^{j-1} \frac{\partial}{\partial X}$ & $(X + t j Y^{j-1}, Y)$ \\
$(\tr X^j)^2$ & $ - 2j (\tr X^j) X^{j-1} \frac{\partial}{\partial Y}$ & $(X, Y - 2 t j (\tr X^j) X^{j-1})$ \\
$(\tr Y^j)^2$ & $2j (\tr Y^j) Y^{j-1} \frac{\partial}{\partial X}$ & $(X + 2 t j (\tr Y^j) Y^{j-1}, Y)$ \\
$\tr X Y$ & $ X \frac{\partial}{\partial X} - Y \frac{\partial}{\partial Y}$ & $(e^t X, e^{-t} Y)$
\end{tabular}

Here, we understand terms like $X^{j-1} \frac{\partial}{\partial Y}$ as entry-wise multiplication of $X^{j-1}$ and $\frac{\partial}{\partial Y}$.
\end{example}

\section{The ring of invariant functions}

Razmyslov \cite{MR0506414} and Procesi \cite{MR419491} proved independently the following about invariant function of tupels of matrices:
\begin{theorem}
Consider the action of $\gl_n(\CC)$ on $\matnn \times \matnn \times \cdots \times \matnn = (\matnn)^m$ by simultaneous conjugation. Then the ring of algebraic invariant functions is generated by
\[
\tr( F_1 \cdot F_2 \cdots F_k )
\]
where each of the $F_1, \dots,  F_k$ is a matrix from one of the $m$ factors and $k \leq n^2 $. Moreover, all relations between the generators are consequences of the Cayley--Hamilton identity.
\end{theorem}

%For the study of the invariant functions on $\widetilde{\camo_n}$, i.e.\ when $m=2$ and under the additional constraint that $\rank([X,Y] + \id) = 1$ we first observe that any invariant function $\widetilde{f}$ on $\widetilde{\camo_n}$ descends to a function $f$ on $\camo_n$. This $f$ extends to $\widehat{f}$ on $(\matnn \times \matnn) /\!/ \gl_n(\CC)$ which in turn lifts to an invariant function $\widetilde{g}$ on $\matnn \times \matnn$. Now, both $\widetilde{f}$ and $\widetilde{g} | \widetilde{\camo_n}$ induce the same function $f$ on the quotient and hence agree on $\widetilde{\camo_n}$. We can therefore assume that any invariant function on $\widetilde{\camo_n}$ is in fact the restriction of an invariant function on $\matnn \times \matnn$. 

\begin{remark} \label{remark: rankcondition}
    Known as Noether's conservation law for Hamiltonian systems with symmetry, see Abraham and Marsden \cite{MR515141}*{Theorem 4.2.2}, the flow $\varphi_t$ induced by any Hamiltonian vector field $V_H$ associated to an invariant Hamiltonian function $H$ preserves the fibers of the moment map $\mu$. In our case, one can verify directly that the commutator $[X, Y]$ is constant along the flow $\varphi_t$ for a Hamiltonian function $H(X,Y)$, which is a product of traces of monomials in $X, Y$. In particular, the rank condition $\mathrm{rank}([X,Y] + \mathrm{id})=1$ is preserved. 
\end{remark}

So far we have not made use of the rank condition on $\widetilde{\mathcal{C}_n}$. In fact, it turns out that this condition is magical in reducing the generating set of the ring of invariant functions on $\widetilde{\mathcal{C}_n}$ to a much simpler subset. Let $$B = [X,Y],\quad A= [X,Y]+\mathrm{id} = B + \mathrm{id}.$$
On $\widetilde{\mathcal{C}_n}$, $A$ is of rank one. The following identity for the trace of a product of matrices, where one of the matrices has rank one, will be useful for the reduction. 

\begin{lemma}\label{lemma: trMANA}
    Let $M, N, C \in \mathrm{Mat}(n \times n; \mathbb{C})$, and assume that $C$ is of rank one. Then
    $$\tr MCNC = \tr MC \cdot \tr NC.$$ 
\end{lemma}
\begin{proof}
    Since $C$ is of rank one, write $C = v w^t$ for some $v, w \in \mathbb{C}^n\setminus \{0\}$. Then
    \begin{align*}
        \tr MCNC &= \tr M vw^t Nvw^t = \tr M v(w^t Nv)w^t \\
        &= (w^t Nv) \tr M v w^t = \tr N C \cdot \tr MC.   \qedhere
    \end{align*}
\end{proof}

We also need the following identities concerning the commutator $B$. 

\begin{lemma}\label{lemma: trXYB}
    For $k,l \in \mathbb{N}_0$ 
    $$ \tr X^k B = 0, \tr Y^l B = 0.$$
    And $\tr XYB = \binom{n}{2}$ on $\widetilde{\mathcal{C}_n}$.
\end{lemma}
\begin{proof}
    
    \begin{align*}
        \tr X^k B = \tr X^k (XY-YX) = \tr X^{k+1}Y - \tr X^k Y X = 0.
    \end{align*}
    Similarly $\tr Y^l B = 0$. For $\tr XYB$, because $A$ is of rank one on $\widetilde{\mathcal{C}_n}$, we can find $v, w \in \mathbb{C}^n\setminus \{0\}$ such that $A = v w^t$. Then with $\tr A = \tr ([X,Y]+\mathrm{id}) = \tr \mathrm{id}=n$,
    \begin{align*}
        A^2 = v w^t v w^t = v (w^t v) w^t = (\tr A) A = nA.     
    \end{align*}
    This implies that 
    \begin{align*}
        \tr B^2 &= \tr A^2 - 2 \tr A + \tr \mathrm{id} \\
        &= n \tr A - 2 \tr A + n \\
        &= n(n-1).
    \end{align*}
    Then from
    $$ \tr XYB = - \tr YXB = - \tr XYB + \tr B^2 $$
    it follows that $\tr XYB = \binom{n}{2}$ on $\widetilde{\mathcal{C}_n}$.
\end{proof}

\begin{definition}
    The \emph{(total) degree} $\deg M$ of a monomial $M(X,Y)$ in $X,Y$ is defined to be the total number of the factors $X$ and $Y$ in $M(X,Y)$. The \emph{degree} $\deg H$ of a trace $H(X,Y) = \tr M(X,Y)$ is the degree $\deg M$ of $M$.  For product of traces, the degree is the sum of degrees of the factors. By the \emph{double degree}, we mean the pair of degrees in $X$ and in $Y$. 
\end{definition}

\begin{lemma} \label{lemma: degtrMNB}
    Let $M = M(X,Y), N= N(X,Y)$ be monomials in $X$ and $Y$. Then on $\widetilde{\mathcal{C}_n}$, $\tr M N B$ can be written as a sum of products of traces of monomials in $X,Y$, where the degree of each summand is bounded by $\deg M + \deg N - 2$.
\end{lemma}
\begin{proof}
    Use induction on $k = \deg M + \deg N$. For $k=1,2$, the induction base is given by Lemma \ref{lemma: trXYB}
    \begin{align*}
        & \tr X B = 0, \, \tr YB = 0 \\
        & \tr XY B = \binom{n}{2}, \, \tr X^2 B = 0, \, \tr Y^2 B = 0. 
    \end{align*}
    For the induction step, first consider $\tr MBNB$. By the identities $B = A - \mathrm{id}$ and $\tr MANA = \tr MA \tr NA$ in Lemma $\ref{lemma: trMANA}$,
    \begin{align*}
        \tr MBNB =& \tr M A N A - \tr MNA - \tr MAN + \tr MN \\
        =& \tr MA \cdot \tr NA - \tr MNB - \tr MBN - \tr MN \\
        =& \tr MB \cdot \tr NB + \tr M \cdot \tr NB + \tr MB \cdot \tr N \\ 
        &+ \tr M \cdot \tr N - \tr MNB - \tr NMB - \tr MN.
    \end{align*}
    By the induction hypothesis, each summand above involves terms of degree bounded by $ \deg M + \deg N$. Hence $\tr MBNB$ is a sum of products of traces, where each summand has degree smaller than or equal to  $\deg M + \deg N$. 
    
    Notice that for $YX$ in $MN = M'YXN'$
    \begin{align} \label{equation: YXswapXY}
        \tr M'YXN'B = \tr M'XYN' B - \tr M'BN'B.
    \end{align}
    Thus the pair $YX$ can be replaced with $XY$, up to an additional term of degree smaller than or equal to $\deg M' +\deg N' = \deg M+\deg N -2$. Continuing swapping $YX$ until we obtain $\tr X^i Y^j B$ for some $i,j \in \mathbb{N},\, i+j = \deg M +\deg N$. For this we take the identity $\tr X^{i+1} B = 0$, which under the Calogero--Moser flow $(X, Y) \mapsto (X + t Y^j, Y)$ becomes 
    $$\tr (X+tY^j)^{i+1} B = 0$$
     since $ B= [X,Y]$ is invariant under the flow by Remark \ref{remark: rankcondition}. The coefficient of $t$ in this new identity must vanish, which gives
    \begin{align} \label{equation: flowXiYj}
        0  = \tr X^i Y^j B + \tr X^{i-1}Y^j X B  + \dots + \tr X Y^j X^{i-1} B + \tr Y^j X^i B
    \end{align}
     Continue to swap $YX$ to $XY$ for $\tr X^{i-m}Y^j X^m B, m = 1, \dots, i$, which yields extra summands of degree up to $i+j -2$ as in Equation \eqref{equation: YXswapXY}. Then 
    \begin{align*}
        \tr X^{i-m}Y^j X^m B = \tr X^i Y^j B + \text{ terms of degree up to } i+j-2
    \end{align*}
    This and Equation \eqref{equation: flowXiYj} imply that 
    \begin{align*}
        0 = (i+1) \tr X^i Y^j B + \text{ terms of degree up to } i+j-2 
    \end{align*}
    hence $\tr X^i Y^j B$ consists of terms of degree bounded by $i+j-2$. This in turn shows that $\tr MNB$ is a sum of products of traces, with each summand of degree smaller than or equal to $ \deg M + \deg N -2$.
\end{proof}

\begin{corollary}\label{corollary: MagicOnCM}
    Let $M = M(X,Y), N= N(X,Y)$ be monomials in $X, Y$. Denote the number of $X$ and $Y$ in $M$ by $i_1$ and $j_1$, respectively. Similary $i_2$ and $j_2$ for $N$. Then the trace $\tr MN$ is equal to $\tr X^i Y^j, \, i=i_1+i_2, j = j_1+j_2$, on $\widetilde{\mathcal{C}_n}$ up to terms of degree smaller than or equal to $i+j-4$. 
\end{corollary}
\begin{proof}
    For one $YX$ in $MN=M'YXN'$, we look at
    \begin{align}
        \tr M'YXN' = \tr M'XYN' - \tr M'BN' = \tr M'XYN' - \tr N'M'B 
    \end{align}
    By Lemma \ref{lemma: degtrMNB}, $\tr N'M'B$ consists of terms of degree $\le i+j-4$. Moving all $X$ to the left by swapping $YX$ to $XY$, we see that the claim is true.
\end{proof}

In this manner, any trace of monomials in $X,Y$ on $\widetilde{\mathcal{C}_n}$ of double degree $(i,j)$, can be written as $\tr X^i Y^j$ plus a polynomial in trace of monomials in $X,Y$, whose degree is bounded by $ i+j-4$. 

\begin{proposition} \label{proposition: ReductionInvGenerator}
    The invariant functions on $\mathrm{Mat}(n \times n ; \mathbb{C}) \times \mathrm{Mat}(n \times n ; \mathbb{C})$, restricted to $\widetilde{\mathcal{C}_n}$, can be generated by $ \mathcal{B}_n  = \{  \tr X^i Y^j : i,j \in \mathbb{N}_0, i + j \leq n^2 \}$.
\end{proposition}
\begin{proof}
    By the invariant theory of $n \times n$ matrices, see \cite{MR419491}*{Theorem 1.3 and Theorem 3.4 (a)} and \cite{MR0506414}*{final remark}, the ring of invariant functions is generated by traces of monomials in $X$ and $Y$ with total degree less than or equal to $n^2 $. We show that any such trace of degree $D$ is contained in the ideal generated by $ \mathcal{B}_n$. For $D \le 2$, we only have $$ f \equiv \text{constant},\tr X, \tr Y, \tr X^2, \tr Y^2, \tr XY. $$ Assume that traces of degree $\le D$ are generated by $\mathcal{B}_n$. If we consider a trace function $\tr MN$ of degree $D+1$ as in Corollary \ref{corollary: MagicOnCM}, we can replace it with $\tr X^i Y^j$ plus terms of degree smaller than or equal to $D-3$, which are generated by $\mathcal{B}_n$ according to the induction hypothesis.
\end{proof}

The result of Proposition \ref{proposition: ReductionInvGenerator} above was also obtained by Etingof and Ginzburg \cite{MR1881922}*{Section 11, p.~322, Remark (ii)} where they make use of the Harish-Chandra homomorphism.

\section{Proof of the Hamiltonian Density Property for $n=2$}
We will show that the Lie algebra generated by a set of invariant functions on $\matnn \times \matnn$, whose corresponding vector fields are complete, contains all monomials in the generators of the ring of invariant functions. To this end, we take first the following set of invariant functions on $\mathrm{Mat}(n \times n; \mathbb{C}) \times \mathrm{Mat}(n \times n; \mathbb{C})$
\begin{align*}
    \mathcal{D}  = \{ \tr X^i, \tr Y^j, (\tr X^i)^2, (\tr Y^j)^2 \colon i,j \in \mathbb{N}\}
\end{align*}
Their corresponding vector fields and flows are given in Example \ref{example: CMHamVectFlow}. %Note that $\tr X^i$ and $\tr Y^j$ correspond to the Calogero--Moser flow $\varphi_t: (X,Y) \mapsto (X, Y+t i X^{i-1})$ and $\psi_t: (X,Y) \mapsto (X+ t j Y^{j-1}, Y)$, respectively. %$(\tr X^i)^2$ and  $(\tr Y^j)^2$ are Hamiltonian shears of $\varphi$ and $\psi$, respectively. WHAT IS A HAMILTONIAN SHEAR?

We need the following formula for our bracket computation. 
\begin{lemma}\label{lemma: commutator-general}
    For $a, b, c, d \in \mathbb{N}_0$, 
    \[\begin{split}
        &\{ \tr X^a Y^b, \tr X^c Y^d \} \\=& \sum_{\substack{1\le p \le a \\ 1 \le q \le d}} \tr X^{p-1} Y^{d-q}X^c Y^{q-1}X^{a-p}Y^b \\ &- \sum_{\substack{1\le r \le b \\ 1 \le s \le c}} \tr Y^{r-1} X^{c-s} Y^d X^{s-1} Y^{b-r} X^a.
    \end{split}\]
\end{lemma}
\begin{proof}
    We make use of the product rule for the Poisson bracket and that $\{X_{ij}, Y_{kl}\} = \delta_{jk}\delta_{li}$
    \begin{align*}
       &\{ \tr X^a Y^b, \tr X^c Y^d \} \\
       =& \sum \{ X_{i_1 i_2}\cdots X_{i_a j_1}Y_{j_1 j_2} \cdots Y_{j_b i_1}, X_{k_1 k_2} \cdots X_{k_c m_1}Y_{m_1 m_2} \cdots Y_{m_d k_1} \} \\
       =& \sum \sum_{p,q} \widehat{X_{i_p i_{p+1}}} \widehat{Y_{m_q m_{q+1}}} \{X_{i_p i_{p+1}}, Y_{m_q m_{q+1}}\}  \\
       & + \sum \sum_{r,s} \widehat{Y_{j_r j_{r+1}}} \widehat{X_{k_s k_{s+1}}} \{Y_{j_r j_{r+1}}, X_{k_s k_{s+1}}\}\\
       =& \sum \sum_{p,q} \widehat{X_{i_p i_{p+1}}} \widehat{Y_{m_q m_{q+1}}} \delta_{i_{p+1} m_q} \delta_{m_{q+1} i_p} \\
       & - \sum \sum_{r,s} \widehat{Y_{j_r j_{r+1}}} \widehat{X_{k_s k_{s+1}}} \delta_{j_{r+1} k_s} \delta_{k_{s+1} j_r} \\
%       =&  \sum \sum_{p,q} X_{i_1 i_2} \cdots X_{i_{p-1} i_p} Y_{i_p m_{q+2}}\cdots Y_{m_d k_1} X_{k_1 k_2} \cdots X_{k_c m_1} Y_{m_1 m_2} \cdots Y_{m_{q-1} i_{p+1}} X_{i_{p+1} i_{p+2}} \cdots X_{i_a j_1} Y_{j_1 j_2} \cdots Y_{j_b i_1} \\ TOO LONG
       =& \sum_{\substack{1\le p \le a \\ 1 \le q \le d}} \tr X^{p-1} Y^{d-q}X^c Y^{q-1}X^{a-p}Y^b \\ &- \sum_{\substack{1\le r \le b \\ 1 \le s \le c}} \tr Y^{r-1} X^{c-s} Y^d X^{s-1} Y^{b-r} X^a
    \end{align*}
    where $\widehat{X_{i_p i_{p+1}}}$ stands for the product of all factors $X$ without $X_{i_p i_{p+1}}$, i.e. 
    $$ \widehat{X_{i_p i_{p+1}}} = X_{i_1 i_2}\cdots X_{i_{p-1} i_p} X_{i_{p+1} i_{p+2}} \cdots X_{i_a j_1} X_{k_1 k_2} \cdots X_{k_c m_1} $$
    and we identify $X_{i_a j_1}$ with $X_{i_a i_{a+1}}$, similarly for other terms. 
\end{proof}

Two simple examples:
\begin{align*}
    \{ \tr X, \tr Y \} &=  \{ \sum_k X_{kk},  \sum_j Y_{jj} \}
    = \sum_{k,j} \delta_{kj} \delta_{jk}  = n \\
    \{  \tr X^2 , \tr Y^2 \} &=  \{ \sum_{i,j} X_{ij} X_{ji} ,   \sum_{p,q} Y_{pq} Y_{qp}  \} \\
    &=  \sum_{i,j ,p,q} 4 X_{ji} Y_{qp} \delta_{jp} \delta_{qi} = 4 \tr XY
\end{align*}

\begin{corollary}\label{lemma: 1factor}
    $\tr X^k Y^l \in \mathrm{Lie}(\mathcal{D})$ for any $k, l \ge 0$.
\end{corollary}
\begin{proof}
    %By Lemma \ref{lemma:completeVFs} we have that $\tr X^k, \tr Y^l \in \mathrm{Lie}(\mathcal{D})$ for any $k,l\ge 1$. 
    By Lemma \ref{lemma: commutator-general}, $\{ \tr X^k, \tr Y^l \} = kl \, \tr X^{k-1} Y^{l-1}$.
\end{proof}

In the following we generate invariant functions of the form $\tr(X^i Y^j) \cdot \tr(X^k Y^l)$, where $i, j, k, l$ are natural numbers and at least one of them is smaller than or equal to 2.

\subsection{Finding $\tr(X^k) \cdot \tr(Y^l)$}%, k,l \geq 0$} 
We start with
\begin{align*} 
    \{ \tr X^k, (\tr Y^l)^2 \} = 2 \tr Y^l \{ \tr  X^{k}, \tr Y^{l} \}  = 2kl \, \tr Y^l \,\tr X^{k-1} Y^{l-1}
\end{align*}
When $l = 1$, we have $\tr Y \, \tr X^k$ in the Lie-closure $\mathrm{Lie}(\mathcal{D})$ of $\mathcal{D}$. Under exchange of $X$ and $Y$, we have $ \tr X \, \tr Y^l$. Building the bracket of these two functions, we have the desired term $\tr X^k \, \tr Y^l$ as a summand:
\begin{align*}
    \{ \tr X^k \, \tr Y, \tr X \, \tr Y^l \} &=  \tr X \tr Y \{ \tr X^k , \tr Y^l \} \\
    & \quad + \tr X^k \tr Y^l \{ \tr Y, \tr X \} \\
    &= kl \tr X \, \tr Y \, \tr X^{k-1} Y^{l-1} - n \tr X^k \, \tr Y^l  
\end{align*}
The first term can be obtained as follows:

Using Corollary \ref{lemma: 1factor}
and $\{  (\tr X)^3  , (\tr Y)^2 \} =  6 n  (\tr X)^2  \tr Y$  we have 
\begin{align*}
    \{ \tr X^k Y^l, (\tr X)^2 \, \tr Y \} = k (\tr X)^2 \, \tr X^{k-1} Y^l - 2l \tr X \, \tr Y \, \tr X^k Y^{l-1}
\end{align*}
The first summand can be obtained from 
\begin{align*}
    \{ \tr XY \tr X^2, \tr Y \} &= \tr Y \tr X^2 + 2 \, \tr XY \tr X \\
    \{ (\tr X)^2, \tr XY \tr X \} &= 2 (\tr X)^3 \\
    \{ (\tr X)^3, \tr X^{k-1} Y^{l+1}  \} &=  3(l+1) (\tr X)^2 \, \tr X^{k-1} Y^{l}
\end{align*} 
hence the second summand $\tr X \, \tr Y \, \tr X^k Y^{l-1}$ is in $\mathrm{Lie}(\mathcal{D})$. So does  $\tr X^k \, \tr Y^l$.

\subsection{Finding $\tr(X^i Y^j) \cdot \tr(X^k Y^l)$}%, i,j,k,l \geq 0$}
With
\begin{align} \label{equation: trYjtrXkYl}
    \{ \tr X^j , \tr X^k \, \tr Y^l \} = jl \tr X^k \, \tr X^{j-1} Y^{l-1} 
\end{align}
and its counterpart $  \tr X^{l-1} Y^{j-1} \, \tr Y^k$ under $X \leftrightarrow Y$, we can build the bracket
\begin{align*}
    \{ \tr X^l , \tr X^kY \, \tr Y^j \} = l \tr X^{l+k-1} \, \tr Y^j + lj \tr X^k Y \, \tr X^{l-1} Y^{j-1} 
\end{align*}
yielding $\tr X^i Y \, \tr X^k Y^l$. Repeating it once more
\begin{align*}
    \{ \tr X^l , \tr X^k Y^2 \, \tr Y^j \} = 2l \tr X^{l+k-1}Y \, \tr Y^j + lj \tr X^k Y^2 \, \tr X^{l-1} Y^{j-1} 
\end{align*}
we also obtain $\tr X^i Y^2 \, \tr X^k Y^l$.

Summarizing, we have the following functions:
\begin{align} \label{2factors}
    \tr X^i Y^j \tr X^k Y^l \in \mathrm{Lie}(\mathcal{D}), \,  i, j, k, l \in \mathbb{N}_0
\end{align}
where at least for one of the four indices $i, j, k, l \le 2$. The case when one of $i,j,k,l$ is zero is covered in Equation \eqref{equation: trYjtrXkYl}.

\subsection{Finding a smaller set of generators} Following the steps in the preceding two subsections, we show that $\mathcal{D}$ actually admits a finite set of generators:
\begin{align*}
	\mathcal{F} = \{  \tr Y, \tr Y^2, \tr X^3, (\tr X)^2 \} 
\end{align*}

First we have
\begin{lemma}
	$\{\tr X^j, \tr Y^j, (\tr X)^2, (\tr Y)^2 \colon j = 1, 2, 3 \}  \subset \mathrm{Lie} (\mathcal{F})$.
\end{lemma}
\begin{proof} It is straightforward to verify that
    \begin{align*}
        \{ \{ \{ \tr X^3 , \tr Y^2 \} , \tr Y^2 \} , \tr Y^2 \} &= 48 \tr Y^3  \\
        \{ \{ (\tr X)^2 , \tr Y^2 \} , \tr Y^2  \} &= 8 (\tr Y)^2   \\
        \{ \tr X^3, \tr Y \} &= 3 \tr X^2  \\
        \{ \tr X^2, \tr Y \} &= 2 \tr X 	 \qedhere
    \end{align*}
    %This proves the inclusion. 
\end{proof}

It is clear that interchanging the roles of $X$ and $Y$ yields an equivalent set of generators. Next, we make the following simplification.
\begin{lemma} \label{lemma: trXtrYuptoDegree3}
     For any $k \in \mathbb{N}$ we have $$\tr X^k, \tr Y^k \in \mathrm{Lie}(\{ \tr X^j, \tr Y^j  \colon j =1, 2, 3  \})$$
\end{lemma}
\begin{proof}
    For $k \le 3$ the statement is trivial, and hence we proceed by induction on $k \ge 3$. The induction step follows from
    \begin{align*}
        \{ \tr X^k , \tr Y^2 \} &= 2k \, \tr X^{k-1}Y \\
        \{ \tr X^3, \tr X^{k-1}Y \} &= 3 \, \tr X^{k+1}
    \end{align*} 
    Similarly for $\tr Y^{k+1}$. 
\end{proof}

In order to make use of the forthcoming Lemma \ref{lemma: hep-hweightnon0}, we shall check that $(\tr XY)^2 \in \mathrm{Lie}(\mathcal{F})$. Starting with
\begin{align*}
    \{ \tr X^3, (\tr Y)^2 \} &= 6 \tr X^2 \tr Y \\
    \{ (\tr X)^2, \tr Y^3 \} &= 6 \tr X \tr Y^2 
\end{align*}
we continue to build bracket between them, which leads us to 
\begin{align*} 
    \{ \tr X^2 \tr Y, \tr X \tr Y^2 \} = 4  \tr XY \tr X \tr Y - n \tr X^2 \tr Y^2
\end{align*}
To conclude that $\tr X^2 \tr Y^2 \in \mathrm{Lie}(\mathcal{F})$, we shall generate the term $\tr XY \tr X \tr Y$. 
\begin{align*}
    \{ \tr X^2, \tr X \tr Y^2 \} &= 4 \tr XY \tr X \\ 
    \{ (\tr X)^2, \tr XY \tr X \} &= 2 (\tr X)^3 \\
    \{ (\tr X)^3, (\tr Y)^2 \} &= 6n (\tr X)^2 \tr Y \\
    \{ \tr XY^2, (\tr X)^2 \tr Y \} &= (\tr X)^2 \tr Y^2 - 4 \tr XY \tr X \tr Y  
\end{align*}
The first term $(\tr X)^2 \tr Y^2$ of the last equation is in $\mathrm{Lie}(\mathcal{F})$ since $\{ (\tr X)^3, \tr Y^3 \} = 9 (\tr X)^2 \tr Y^2$. Thus $\tr XY \tr X \tr Y$ is in $\mathrm{Lie}(\mathcal{F})$. After obtaining $\tr X^2 \tr Y^2$, we proceed with
\begin{align*}
    \{ \tr X^2 \tr Y^2, \tr Y^2 \} &= 4 \tr XY \tr Y^2 \\
    \{ \tr X^2, \tr XY \tr Y^2 \} &= 2 \tr X^2 \tr Y^2 + 4 (\tr XY)^2
\end{align*}
which shows that $(\tr XY)^2 \in \mathrm{Lie}(\mathcal{F})$. Hence we can obtain $(\tr X^k)^2$ for any $k \in \mathbb{N}$ from $\tr X^k$ and $(\tr XY)^2$ in the same manner as in Lemma \ref{lemma: hep-hweightnon0}. Thus, we have proved the following Lemma.

\begin{lemma}
    $\mathcal{D} \subset \mathrm{Lie}(\mathcal{F})$.
\end{lemma}

\subsection{Brackets between generating invariant functions}
Let 
\begin{align*}
    a = \tr X, \, b = \tr Y, \, c = \frac{1}{2}\tr X^2, \, d = \frac{1}{2}\tr Y^2, \, e = \tr XY.
\end{align*}
Let $\mathcal{A} = \{ a, b, c, d, e \}$. The brackets of its elements are:
\begin{align*}
    &\{ a, b \} = n,  \, \{ a, c \} = 0, \, \{a, d \} = b, \, \{a, e \} = a, \, \{ b, c \} = -a, \\ & \{b, d \} = 0, \, \{b, e \} = -b, \, \{c, d \} = e, \,   \{c, e \} = 2 c, \, \{d, e \} = -2d.
\end{align*}
These five elements form a closed set under brackets. Observe that for any element $h$ from $\mathcal{A}$ we have that 
\begin{align} \label{weight}
    \{ h, e \} = w_h h
\end{align}
where $w_h$ is an integer, called the \emph{weight} of $h$. Then different elements carry different weights, and the weights of $a, b, c, d, e$ are $1, -1, 2, -2, 0$, respectively. In particular, $e$ is the only element of $\mathcal{A}$ with weight 0. For the trace of  a monomial in $X,Y$, the weight is well-defined, since taking bracket with $ e= \tr XY $ by Lemma \ref{lemma: commutator-general} amounts to assigning each factor $X$ in the monomial weight $1$, each factor $Y$ weight $-1$, and adding them up. 

\begin{lemma}\label{lemma: productweightsum}
    Let $h_1, h_2$ be traces of monomials in $X,Y$. Then $w_{h_1 h_2} = w_{h_1} + w_{h_2} = w_{\{h_1, h_2 \}}$.
\end{lemma}
\begin{proof}
    This follows from the Leibniz rule and the Jacobi identity of the Poisson bracket.
\end{proof}

\begin{proposition}\label{proposition: monomial-abcde}
    On $\mathrm{Mat}(n \times n; \mathbb{C}) \times \mathrm{Mat}(n \times n; \mathbb{C})$, the Lie-closure $\mathrm{Lie}(\mathcal{F})$ of
    \begin{align*}
        \mathcal{F} = \{  \tr Y, \tr Y^2, \tr X^3, (\tr X)^2 \} 
    \end{align*}
    contains all products of the form
    \begin{align}
        h_1 h_2 \cdots h_m
    \end{align}
    where $m \in \mathbb{N}, h_i \in \mathcal{A}, i = 1, 2, \dots, m$. 
\end{proposition}

\begin{remark}
Since we only need $a, b, c, d, e$ to generate the ring of invariant functions for $n=2$, see \cite{MR0223379}*{Theorem 5}, Proposition \ref{proposition: monomial-abcde} proves already the Hamiltonian holomorphic density property for $n=2$. However, all the results in this section are also valid for $n>2$. %and will be used in the next section as well.
\end{remark}

Note that $\mathcal{A} \subset \mathrm{Lie}(\mathcal{F})$ since $e = \{c, d\}$. We first prove the statement for $m = 3$. The case $m = 1$ is obvious and $m = 2$ is covered by Equation \eqref{2factors} in the preceding subsection. 

Observe that
\begin{align*} 
    \{ he, e^2 \} = 2 w_h h e^2, 
\end{align*}
hence $h e^2 \in \mathrm{Lie}(\mathcal{F})$ when $h \neq e$. Then consider
\begin{align*}
    \{ c^2, d^2 \} &= 4 cde, \\
    \{ c e^2, d \} &= 4 cde +  e^3 . 
\end{align*}
These brackets induce that $e^3 \in \mathrm{Lie}(\mathcal{F})$. Then $h e^2 \in \mathrm{Lie}(\mathcal{F})$ for any $h$ from $\mathcal{A}$. Now take any $h_1, h_2, h_3 \in \mathcal{A}$,
\begin{align*}
    \{ h_1 e^2, h_2 \} = \{h_1, h_2 \} e^2 - 2 w_{h_2} h_1 h_2 e 
\end{align*}
Since $\{h_1, h_2 \} e^2 \in \mathrm{Lie}(\mathcal{F})$, $h_1 h_2 e$ is in the algebra when $w_{h_2} \neq 0$. Otherwise $h_2 = e$, since it is the only element in $\mathcal{A}$ of weight $0$ and we are back to $h_1 e^2 \in \mathrm{Lie}(\mathcal{F})$.  Continue with the same strategy
\begin{align*}
    \{ h_1 h_2 e, h_3 \} &= \{h_1, h_3 \} h_2 e + \{h_2, h_3 \} h_1 e - w_{h_3} h_1 h_2 h_3
\end{align*}
this covers the case $m = 3$. 
% REMARK. For even $m$ the above induction works well, while for odd $m$ the induction fails to find $e^m$ or $c^{m/2} d^{m/2}e$.

In a similar spirit, the following lemma makes use of $e^2 \in \mathrm{Lie}(\mathcal{F})$ to include exponents of an element of nonzero weight in $\mathrm{Lie}(\mathcal{F})$.

%\begin{align}
   % \{ c^k e^{2p-1}, de \} &= \{c^k, d \} \,  e^{2p} + (2k + 2(2p-1)) \, c^k d e^{2p-1} \\
   % &= 4k \,c^{k-1} e^{2p+1} + (2k + 2(2p-1)) \,c^k d e^{2p-1},
%\end{align}

\begin{lemma}\label{lemma: hep-hweightnon0}
    Let $h \in \mathrm{Lie}(\mathcal{F})$ with weight $w_h \neq 0$, let $k \ge 1, p \ge 0$.  Then $h^k e^p \in \mathrm{Lie}(\mathcal{F})$.
\end{lemma}
\begin{proof}
    First we show that $h e^p \in \mathrm{Lie}(\mathcal{F})$ by induction. For $p=0$ it is obvious. The induction step is provided by $ \{ h e^p, e^2 \} = 2 w_h h e^{p+1}$.  Next we show $h^k e^p \in \mathrm{Lie}(\mathcal{F})$ by induction on $k$. The case $k=1$ has been cleared. The induction step is given by $\{ h^k e^p, h \} = - p w_h \, h^{k+1} e^{p-1} $.
\end{proof}

Applying Lemma \ref{lemma: hep-hweightnon0} to $h \in \{ a,b,c,d \}$, we obtain 
\begin{align*}
    a^i, b^j, c^k, d^l \in \mathrm{Lie}(\mathcal{F}), \quad  i,j,k,l \in \mathbb{N}
\end{align*}
without using Lemma \ref{lemma:completeVFs}. These are now used to generate the following terms:
\begin{align}
    \{ a^i, b^j \} &= n i j \, a^{i-1} b^{j-1} \label{aibj} \\ 
    \{ c^k, d^l \} &=  k l \, c^{k-1} d^{l-1} e  \label{ckdle}
\end{align}

\begin{corollary} \label{corollary: ckep}
    Let $k,l \ge 1, p \ge 0$. Then $c^k e^p, d^l e^p \in \mathrm{Lie}(\mathcal{F})$.
\end{corollary}
\begin{proof} Since $c,d \in \mathrm{Lie}(\mathcal{F})$ are of nonzero weight, the claim follows from Lemma \ref{lemma: hep-hweightnon0}. 
\end{proof}

\begin{remark}
    At this stage, it is not clear if $e^p$ is in $\mathrm{Lie}(\mathcal{F})$ for $p \ge 4$. To show this, functions of higher degree in $X$ are needed here. 
\end{remark}

\begin{lemma}\label{lemma: thirdsummand}
    Let $i,j \in \mathbb{N}$ such that $ i - j -1 \neq 0$. Then $(\tr X^{i}Y) b^j e \in \mathrm{Lie}(\mathcal{F})$.
\end{lemma}
\begin{proof}
    By Lemma \ref{lemma:completeVFs}, $b^{j+1} \in \mathrm{Lie}(\mathcal{F})$ and by Corollary \ref{lemma: 1factor}, $\tr X^{i+1}Y \in \mathrm{Lie}(\mathcal{F})$, therefore
    $$\{ \tr X^{i+1} Y, b^{j+1} \} = (i+1)(j+1) (\tr X^{i}Y) b^{j} \in \mathrm{Lie}(\mathcal{F}).$$ Since $(\tr X^{i}Y) b^j$ carries weight $i-1-j \neq 0$ by Lemma \ref{lemma: commutator-general} and Lemma \ref{lemma: productweightsum}, $(\tr X^{i}Y) b^j e \in \mathrm{Lie}(\mathcal{F})$ by Lemma \ref{lemma: hep-hweightnon0}.
\end{proof}

\begin{lemma} \label{lemma: trXi+mbid}
    Let $i \ge 0$ and $j \in \{ 3, 4\}$. Then $(\tr X^{i+j}) b^i d \in \mathrm{Lie}(\mathcal{F})$.
\end{lemma}
\begin{proof}
    We induct on $i$. In the case $i=0$, $(\tr X^j) d \in \mathrm{Lie}(\mathcal{F})$ by Equation \eqref{2factors}. For the induction step, use the bracket
    \begin{align*}
        & \{ (\tr X^{i+j}) e, b^i d \} \\ =& (i+2) (\tr X^{i+j}) b^i d + i(i+j) (\tr X^{i-1+j}) b^{i-1}de \\
        &+ (i+j) (\tr X^{i-1+j}Y) b^i e
    \end{align*}
    where $(\tr X^{i+j}) e \in \mathrm{Lie}(\mathcal{F})$ since $\tr X^{i+j}$ has nonzero weight. Notice that $b^i d$ can be generated as
    $$ \{ \tr XY^2, b^{i+1} \} = (i+1) b^i d.$$ 
    By the induction hypothesis, $(\tr X^{i-1+j}) b^{i-1}d \in \mathrm{Lie}(\mathcal{F})$, which by Lemma \ref{lemma: productweightsum} carries weight $i-1+j-(i-1)-2 = j-2 \ge 1$. Hence also $(\tr X^{i-1+j}) b^{i-1}d e \in \mathrm{Lie}(\mathcal{F})$ by Lemma \ref{lemma: hep-hweightnon0}. The third summand is in $\mathrm{Lie}(\mathcal{F})$ by Lemma \ref{lemma: thirdsummand}. Therefore the first summand $(\tr X^{i+j}) b^i d \in \mathrm{Lie}(\mathcal{F})$, which concludes the proof. 
\end{proof}
Since $(\tr X^{i+3}) b^i d$ is of nonzero weight, by Lemma \ref{lemma: hep-hweightnon0} we have 
\begin{corollary}\label{corollary: trXi+3bidej}
    $(\tr X^{i+3}) b^i d e^j \in \mathrm{Lie}(\mathcal{F})$ for any $i, j \ge 0$. 
\end{corollary}

Now we are ready to show that $e^{p} \in \mathrm{Lie}(\mathcal{F})$ for $p \ge 4$. 
\begin{lemma}\label{lemma: ep}
    $e^p \in \mathrm{Lie}(\mathcal{F})$ for any $p \ge 0$.
\end{lemma}
\begin{proof}
    Since $e, e^2, e^3 \in \mathrm{Lie}(\mathcal{F})$, let $p \ge 4$.
    By Lemma \ref{lemma: trXi+mbid}, $(\tr X^{p+2}) b^{p-2} d \linebreak \in \mathrm{Lie}(\mathcal{F})$. Use
    \begin{align*}
        \{ (\tr X^{p+2}) b^{p-2} d, b^2 \} = 2 (p+2) (\tr X^{p+1}) b^{p-1} d 
    \end{align*}
    which gives $\tr(X^{p+1}) b^{p-1} d \in \mathrm{Lie}(\mathcal{F})$. Take it as the induction base $i = 0$ for the claim that 
    $$(\tr X^{p-i+1}) b^{p-i-1} d e^{i} \in \mathrm{Lie}(\mathcal{F})$$
     for $i = 0, 1, \cdots, p-1$. Suppose that $(\tr X^{p-i+2}) b^{p-i} d e^{i-1} \in \mathrm{Lie}(\mathcal{F})$.
    
    Since $\tr(X^{p-i+2}) b^{p-i-1} d e^{i} \in \lie(\mathcal{F})$ by Corollary \ref{corollary: trXi+3bidej}, the bracket
    \begin{align*}
        &\{ (\tr X^{p-i+2}) b^{p-i-1} d e^{i}, b \} \\ 
        =& \, (\tr X^{p-i+2}) b^{p-i-1} d \{ e^i, b\} + b^{p-i-1} d e^{i} \{ \tr X^{p-i+2}, b \}  \\
        =& \, i\, (\tr X^{p-i+2}) b^{p-i} d e^{i-1} + (p-i+2)\,  (\tr X^{p-i+1}) b^{p-i-1} d e^{i} 
    \end{align*}
    together with the induction hypothesis, yields $(\tr X^{p-i+1}) b^{p-i-1} d e^{i} \in \mathrm{Lie}(\mathcal{F})$. With $i = p-1$ we have $(\tr X^2) d e^{p-1} \in \mathrm{Lie}(\mathcal{F})$,
    that is, $c d e^{p-1} \in \mathrm{Lie}(\mathcal{F})$. Finally, since $ce^p \in \mathrm{Lie}(\mathcal{F})$ by Corollary \ref{corollary: ckep}, using $\{ c e^p, d \} = 2p \, cd e^{p-1} + e^{p+1}$ completes the proof. 
\end{proof}

Now we have an improved version of Corollary \ref{corollary: ckep}:

\begin{lemma}\label{lemma: ckep-general}
    $c^k e^p, d^l e^p \in \mathrm{Lie}(\mathcal{F})$ for any $k,l, p \ge 0$.
\end{lemma}

Here we continue the task to generate general monomials. 
\begin{lemma}\label{lemma: abc}
    $a^i b^j c^k \in \mathrm{Lie}(\mathcal{F})$ for any $i, j, k \ge 0$.
\end{lemma}
\begin{proof}
    From Equation \eqref{aibj} we have the case $k=0$ as the induction base. The induction step follows from the bracket
    \begin{align*} 
        \{ a^i c^{k+1}, b^j \} = n i j \, a^{i-1} b^{j-1} c^{k+1} +  (k+1) j \, a^{i+1} b^{j-1} c^{k}. 
    \end{align*}
    We remark that $a^i c^{k+1}$ can be generated as $\{ a^i, c^{k+1} e \} = i a^i c^{k+1}$.
\end{proof}

\begin{lemma}\label{lemma: abce}
     $a^i b^j c^k e^p \in \mathrm{Lie}(\mathcal{F})$ for any $i, j, k, p \ge 0$.
\end{lemma}
\begin{proof}
    We induct on $p$. The case $p=0$ is covered by Lemma \ref{lemma: abc}. For the induction step, consider
    \begin{align*} %\label{equation: abcde}
        \{ a^i b^j , c^k e^p \} &= \{ a^i, e^p \} b^j c^k + \{ b^j, e^p \} a^i c^k + \{ b^j, c^k \} a^i e^p \\
        &= (i-j) p \, a^i b^j c^k e^{p-1} -  j k \, a^{i+1} b^{j-1} c^{k-1} e^p
    \end{align*}
    which covers $a^i b^j c^k e^p$ with $i \ge 1$. Here Equation \eqref{aibj} and Lemma \ref{lemma: ckep-general} are used to make sure that $a^i b^j, c^k e^p \in \mathrm{Lie}(\mathcal{F})$. To obtain $b^j c^k e^p \in \mathrm{Lie}(\mathcal{F})$, take
    \begin{align*}
        \{ b^j, c^k e^p \} = - jk \, a b^{j-1} c^{k-1} e^p -  jp \, b^j c^k e^{p-1}
    \end{align*}
    where the first summand is already in $\mathrm{Lie}(\mathcal{F})$.
\end{proof}

Finally we generate all monomials in $a,b,c,d,e$.
\begin{lemma}
    $a^i b^j c^k d^l e^p \in \mathrm{Lie}(\mathcal{F})$ for any $i,j,k,l,p \ge 0$. 
\end{lemma}
\begin{proof}
    We induct on $l$. Lemma \ref{lemma: abce} gives the base case $l = 0$. The induction step follows from
    \begin{align*}
        &\{ a^i b^j c^k e^p, d^l \} \\ =& \{ a^i, d^l \} b^j c^k e^p + \{ c^k , d^l \} a^i b^j e^p + \{ e^p, d^l \} a^i b^j c^k \\
        =& \, i l \, a^{i-1} b^{j+1} c^k d^{l-1} e^p + k l \, a^i b^j c^{k-1} d^{l-1} e^{p+1} + 2 p l \, a^i b^j c^k d^l e^{p-1}.\qedhere
    \end{align*}
\end{proof}

This finishes the proof of Proposition \ref{proposition: monomial-abcde}. We conclude this section with a remark that for $n=2$ the ring of invariant functions needs only the five generators $a, b, c, d, e$; While for $n > 2$ more generators show up and there the reduction using the rank condition on $\widetilde{\cont_n}$ as in Corollary \ref{corollary: MagicOnCM} plays a crucial role.

\section{Proof of the Hamiltonian Density Property for $n > 2$}

\begin{theorem} \label{theorem: mfactors}
    On $\widetilde{\mathcal{C}_n}$, the Lie-closure $\mathrm{Lie}(\mathcal{F})$ of
    \begin{align*}
        \mathcal{F} = \{  \tr Y, \tr Y^2, \tr X^3, (\tr X)^2 \} 
    \end{align*}
    contains all products of the form
    \begin{align*}
        h_1 h_2 \cdots h_m
    \end{align*}
    where $m \in \mathbb{N}_0, h_k \in \mathcal{B} = \{  \tr X^p Y^q : p,q \in \mathbb{N}_0 \}, k = 1, 2, \dots, m$. 
\end{theorem}

\begin{remark}
    As in Lemma \ref{lemma: hep-hweightnon0}, $\tr Y, \tr X^3, e^2 = (\tr XY)^2$ are sufficient to generate $(\tr X)^2$. Hence we could also take the generating set 
    \begin{align*}
        \mathcal{F}' = \{\tr Y, \tr Y^2, \tr X^3, (\tr XY)^2  \}
    \end{align*}
    However, the flow of $e^2$ is not algebraic but holomorphic, see Example \ref{example: CMHamVectFlow}.
\end{remark}

\begin{remark}
Etingof and Ginzburg proved that $\{ \tr Y^2, \tr X^k \;:\; k \geq 0 \}$ generates the ring of function on $\camo_n$ as a Poisson algebra \cite{MR1881922}*{Section 11, p.~321, Equation (11.33)}. Our Theorem \ref{theorem: mfactors} above improves this result in two ways:  First, we do not use the associative multiplication of the Poisson algebra but only the Lie algebra structure. Second, we only need four generators instead of a countable family.
\end{remark}

To prove Theorem \ref{theorem: mfactors}, we go by induction on the total degree $D$ of the product $h_1 h_2 \cdots h_m$. For $D = 0$, we have the constant-valued functions. Suppose that 
\begin{align}\label{indhyp}
    h_1 h_2 \cdots h_m \in \mathrm{Lie}(\mathcal{F}) \text{ up to degree } D.
\end{align}
We use the notation $$f \thicksim g$$ if $f$ and $g$ are of the same degree and are identical up to terms of degree $\deg f -4$. The missing terms are in $\mathrm{Lie}(\mathcal{F})$ if the degree is bounded by $D$, which implies that $\deg f = \deg g \le D+4$. The following example, though not necessary for the proof of Theorem \ref{theorem: mfactors}, illustrates the power of the reduction using the rank condition, which makes it possible to lift the condition on the indices in Equation \eqref{2factors}, with an extra condition on the degree.

\begin{lemma}
    Assume that the induction hypothesis \eqref{indhyp} holds, and let $i,j,k,l \in \mathbb{N}_0, i+j+k+l \le D+4$. Then $$ \tr X^i Y^j \tr X^k Y^l \in \mathrm{Lie}(\mathcal{F}).$$
\end{lemma}
\begin{proof}
    From Equation \eqref{2factors}, $\tr X^k Y \tr X^i Y^j \in \mathrm{Lie}(\mathcal{F})$. Apply Lemma \ref{lemma: commutator-general} and Corollary \ref{corollary: MagicOnCM}
    \begin{align*}
        &\{ \tr X^k Y \tr X^i Y^j, \tr Y^l \} \\
        =& \,l \tr X^i Y^j \sum_{1 \le a \le k} \tr X^{a-1} Y^{l-1} X^{k-a} Y \\
         &+ l \tr X^k Y \sum_{1 \le b \le i} \tr X^{b-1} Y^{l-1}X^{i-b}Y^j \\
        =& \, kl \tr X^i Y^j \tr X^{k-1} Y^{l} + l \tr X^i Y^j \cdot ( \text{terms of degree} \le k+l-5 )  \\
        &+ il \tr X^k Y \tr X^{i-1}Y^{j+l-1} \\
        &+ l \tr X^k Y \cdot (\text{terms of degree} \le i+l+j-6) 
    \end{align*}
    The third term is in $\mathrm{Lie}(\mathcal{F})$, while the second and fourth terms are of degree $\le i+j+k+l-5$. In the above notation
    \[\begin{split}
        \{ \tr X^k Y \tr X^i Y^j, \tr Y^l \} \thicksim \\
        kl \tr X^i Y^j \tr X^{k-1} Y^{l} + il \tr X^k Y \tr X^{i-1}Y^{j+l-1} 
    \end{split}\]
    For $k,l \ge 1$, if $i+j+k+l-5 \le D$, then $\tr X^i Y^j \tr X^{k-1} Y^{l} \in \mathrm{Lie}(\mathcal{F})$. The case $l = 0$ can be seen by interchanging the roles of $ X$ and $ Y$. 
\end{proof}

The proof of Theorem \ref{theorem: mfactors} starts with the generating set $\mathcal{F}$, then Corollary \ref{lemma: 1factor} provides the first step for $\tr X^p Y^q$ and we continue with taking in factors of the form $\tr X^i$. 

\begin{lemma} \label{lemma: trXYPRODtrX}
    Assume that the induction hypothesis \eqref{indhyp} holds, and let $p, q, l \in \mathbb{N}_0, i_1, \dots, i_l  \in \mathbb{N}_0$ with $p+q+\sum i_k \le D+4$. Then 
    \begin{align*}
        \tr X^p Y^q  \prod_{k=1}^l \tr X^{i_k} \in \mathrm{Lie}(\mathcal{F}). 
    \end{align*}
\end{lemma}
\begin{proof}
    Use induction on $l$. The cases $l = 0$ is Corollary \ref{lemma: 1factor}. Assume the claim for $l$. First take $p+q+ \sum i_k \le D+4$, then $\tr X^p Y^q  \prod_{k=1}^l \tr X^{i_k} \in \mathrm{Lie}(\mathcal{F})$ by the induction assumption. Applying Lemma \ref{lemma: commutator-general} and Corollary \ref{corollary: MagicOnCM} we obtain
    \begin{align*}
        &\{ (\tr X^j)^2, \tr X^p Y^q  \prod_{k=1}^l \tr X^{i_k} \} \\
        =& \, 2 j q \tr X^{p+j-1} Y^{q-1} \tr X^j \prod_{k=1}^l \tr X^{i_k} \\ &+ 2 j  \cdot (\text{terms of degree} \le p+q+j-6)  \cdot \tr X^j \prod_{k=1}^l \tr X^{i_k}
    \end{align*}
    If $2j+p+q+\sum i_k \le D+6$, the last term is in $\mathrm{Lie}(\mathcal{F})$ by \eqref{indhyp}. Then for $j,q \ge 1$, 
    \begin{align*} %\label{p+j-1VSj}
        \tr X^{p+j-1} Y^{q-1} \tr X^j \prod_{k=1}^l \tr X^{i_k} \in \mathrm{Lie}(\mathcal{F})
    \end{align*}
    when the sum of exponents is smaller than or equal to $D+4$. However, observe that the exponent $p+j-1$ plus one is greater than or equal to the exponent $j$. Renaming the exponents $p+j-1 \to p, q-1 \to q, j \to i_{l+1}$, we see that this is equivalent to 
    \begin{align} \label{oneOFik<p}
        \tr X^{p} Y^{q} \prod_{k=1}^{l+1} \tr X^{i_k} \in \mathrm{Lie}(\mathcal{F})
    \end{align}
    where the sum of exponents is smaller than or equal to $D+4$ and one of the exponents $i_k \le p+1$. 
    
    The opposite case is when all $i_k \ge p+2, k = 1, \dots, l+1$. For this we consider for $ 1 \le q $ and $ 2 \le p+2 \le j, i_k$
    \begin{align*}
        &\{ \tr X \tr X^j, \tr X^p Y^q \prod_{k=1}^{l} \tr X^{i_k} \} \\
        =& \, q \tr X^j \tr X^{p}Y^{q-1} \prod_{k=1}^{l} \tr X^{i_k} + jq \tr X \tr X^{p+j-1} Y^{q-1} \prod_{k=1}^{l} \tr X^{i_k} \\
        &+ (\text{terms of degree} \le p+q+j+\sum i_k -5)
    \end{align*}
    Note that $\{ (\tr X)^2, \tr X^j Y \} = 2 \tr X \tr X^j$. Since $ 1 \le p + j $, the second summand is in $\mathrm{Lie}(\mathcal{F})$ by \eqref{oneOFik<p}. Then so is the first summand, hence 
    \begin{align*}
        \tr X^{p} Y^{q} \prod_{k=1}^{l+1} \tr X^{i_k} \in \mathrm{Lie}(\mathcal{F})
    \end{align*}
    when all $i_k \ge p+2, q \ge 1$ and the sum of exponents is smaller than or equal to $D+4$. This completes the proof.
\end{proof}

Next we aim for more factors of the form $\tr X^p Y^q$. 
\begin{lemma} \label{lemma: trXiatrXpbYqb}
    Assume that the induction hypothesis \eqref{indhyp} holds, and let $l \in \mathbb{N}, k \in \mathbb{N}_0$, $i_1,\dots,i_l, p_1, \dots, p_k,$ $q_1,\dots,q_k \in \mathbb{N}_0$ with $\sum_a i_a+\sum_b p_b + q_b \le D+4$ and at least one of $i_a$ is positive. Then 
    \begin{align*}
        \left( \prod_{a=1}^l \tr X^{i_a} \right) \prod_{b=1}^k \tr X^{p_b}Y^{q_b} \in \mathrm{Lie}(\mathcal{F}).
    \end{align*}
\end{lemma}
\begin{proof}
    Use induction on $k$. Notice that $l \ge 1$, one of $i_a$, say $i_1 \ge 1$, and Lemma \ref{lemma: trXYPRODtrX} is the induction base $k=1$. For the induction step, if $q_c = 0$ for some $c \in \{ 1, \dots, k\}$, then it follows from the induction assumption that 
    \begin{align*}
        &\tr X^p Y^q \left(\prod_{a=1}^l \tr X^{i_a}\right) \prod_{b=1}^{k} \tr X^{p_b}Y^{q_b}  \\
        &= \left(\prod_{a=1}^l \tr X^{i_a}\right) \tr X^{p_c} \tr X^p Y^q \prod_{\substack{b=1 \\b\neq c}}^k \tr X^{p_b}Y^{q_b} \in \mathrm{Lie}(\mathcal{F})
    \end{align*}
    when the degree condition is satisfied. It thus remains to consider the case $q_b \ge 1$ for all $b=1, \dots, k$. For $i, j \ge 1$
    \begin{align} \label{equation: prodtrXiprodtrXpYq}
        &\{  \tr X^i \prod_{b=1}^k \tr X^{p_b}Y^{q_b}, \tr Y^j  \prod_{a=1}^l \tr X^{i_a} \} \\
        \thicksim& \, ij \tr X^{i-1}Y^{j-1} \left( \prod_{a=1}^l \tr X^{i_a} \right)  \prod_{b=1}^k \tr X^{p_b}Y^{q_b} \notag \\
        &+ j \tr X^i \left( \prod_{a=1}^l \tr X^{i_a} \right) \sum_{b=1}^k p_b \tr X^{p_b-1}Y^{q_b+j-1} \prod_{c \neq b} \tr X^{p_c}Y^{q_c} \notag \\
        &-  \tr X^i \tr Y^j \sum_{a=1}^l \sum_{b=1}^k i_a q_b \tr X^{p_b+i_a-1}Y^{q_b-1} \left(\prod_{d \neq a} \tr X^{i_d}\right) \prod_{c \neq b} \tr X^{p_c}Y^{q_c} \notag
    \end{align}
    Each summand in the third line of Equation \eqref{equation: prodtrXiprodtrXpYq} has $k$ factors $\tr X^p Y^q$ and is in $\mathrm{Lie}(\mathcal{F})$ by the induction assumption, if the degree
    \begin{align} \label{equation: degreecondition}
        i + j -2 + \sum_a i_a + \sum_b p_b + q_b \le D+4. 
    \end{align}
    To ensure that the two terms in the Poisson bracket in Equation \eqref{equation: prodtrXiprodtrXpYq} are in $\mathrm{Lie}(\mathcal{F})$, we need
    \begin{align*}
        i+ \sum_b p_b+q_b \le D+4, \quad j + \sum_a i_a \le D+4
    \end{align*}
    which are valid when the inequality \eqref{equation: degreecondition} holds, since $i,j,q_b,i_1 \ge 1$. Therefore we only need to impose one degree condition \eqref{equation: degreecondition}.
    
    Next, for
    \begin{align*}
        (q_1, \dots, q_k ) = (1,  \dots, 1)
    \end{align*}
    each summand in the last line of Equation \eqref{equation: prodtrXiprodtrXpYq} has $k$ factors $\tr X^p Y^q$ (including $\tr Y^j$) and is in $\mathrm{Lie}(\mathcal{F})$. Since $i,j \ge 1$, this implies that the term in the second line is in $\mathrm{Lie}(\mathcal{F})$ with this choice of $q_b=1$. Suppose that 
    \begin{align*}
        \tr X^{i-1}Y^{j-1} \left( \prod_{a=1}^l \tr X^{i_a} \right)  \prod_{b=1}^k \tr X^{p_b}Y^{q_b} \in \mathrm{Lie}(\mathcal{F})
    \end{align*}
    for all $(q_1, \dots, q_k ) \prec (s_1,  \dots, s_k)$ in the lexicographic order. Then it is also in $\mathrm{Lie}(\mathcal{F})$ with $(q_1, \dots, q_k )=(s_1,  \dots, s_k)$, since the terms in the last line of Equation \eqref{equation: prodtrXiprodtrXpYq} have $(q_1, \dots, q_k ) \prec (s_1,  \dots, s_k)$ as exponents of $Y$. An induction on $(q_1, \dots, q_k )$ in the lexicographic order shows that 
    \begin{align*}
        \tr X^{i-1}Y^{j-1} \left( \prod_{a=1}^l \tr X^{i_a} \right)  \prod_{b=1}^k \tr X^{p_b}Y^{q_b} \in \mathrm{Lie}(\mathcal{F})
    \end{align*}
    for any $(q_1, \dots, q_k)$ and with total degree smaller than or equal to $D+4$, which is equivalent to our claim with $l+1$ factors. This completes the proof.
\end{proof}

\begin{lemma} \label{lemma: nfactors}
    Assume that the induction hypothesis \eqref{indhyp} holds, and let $m \in \mathbb{N}_0, p_1, \dots, p_m, q_1,\dots,q_m \in \mathbb{N}_0$ with $\sum_k p_k + q_k \le D+4$. Then 
    \begin{align*}
          \prod_{k=1}^m \tr X^{p_k}Y^{q_k} \in \mathrm{Lie}(\mathcal{F}).
    \end{align*}
\end{lemma}
\begin{proof}
    By Lemma \ref{lemma: trXiatrXpbYqb}, if $i+ \sum_k p_k+q_k \le D+4$, $\tr X^i \prod_k \tr X^{p_k}Y^{q_k} \in \mathrm{Lie}(\mathcal{F})$. Then
    \begin{align*}
        &\{ \tr X^i \prod_{k=1}^{m} \tr X^{p_k}Y^{q_k}, \tr Y^j \} \\
        \thicksim& \, i j \tr X^{i-1} Y^{j-1} \prod_{k=1}^{m} \tr X^{p_k}Y^{q_k} \\
        &+ j \tr X^i \sum_{k=1}^{m} p_k \tr X^{p_k-1}Y^{j-1+q_k} \prod_{a \neq k} \tr X^{p_a}Y^{q_a}
    \end{align*}
    The last term is in $\mathrm{Lie}(\mathcal{F})$ when its degree is smaller than or equal to $D+4$. Hence for $i,j \ge 1$, 
    \begin{align*}
        \tr X^{i-1} Y^{j-1} \prod_{k=1}^{m} \tr X^{p_k}Y^{q_k} \in \mathrm{Lie}(\mathcal{F})
    \end{align*}
    when $i -1 + j - 1 + \sum_k p_k+q_k \le D+4$, which is equivalent to our claim with $m+1$ factors. 
\end{proof}

\begin{proof}[Proof of Theorem \ref{theorem: mfactors}]
Lemma \ref{lemma: nfactors} gives the induction step from degree $D$ up to  degree $D+4$ and hence finishes the proof.
\end{proof}

\begin{proof}[Proof of Theorem \ref{thm-hamiltonDP}]
Proposition \ref{proposition: ReductionInvGenerator} and Theorem \ref{theorem: mfactors} established that all algebraic invariant functions on $\widetilde{\camo_n}$ are contained in the Lie algebra that is generated by the Hamiltonian functions in 
    \begin{align*}
        \mathcal{F} = \{  \tr Y, \tr Y^2, \tr X^3, (\tr X)^2 \} 
    \end{align*}
Hence, on $\camo_n$ we obtain all holomorphic Hamiltonian by taking limits. This proves the Hamiltonian holomorphic density property.
\end{proof}

\begin{proof}[Proof of Theorem \ref{thm-camo-autos}]
We apply the Anders\'en--Lempert Theorem, i.e.\ Theorem \ref{thm-AL} and the fact that $\camo_n$ has the Hamiltonian holomorphic density property with $\mathcal{F}$ as above being the generators. Then the corresponding flow maps generate the identity component of the group of symplectic holomorphic automorphisms: In Theorem \ref{thm-AL}, choose $\Phi_t$ to be a path that connects any given symplectic holomorphic automorphism to the identity and use the conclusion in the last paragraph of said theorem.
\end{proof}

\section*{Acknowledgements}
The authors would like to thank Frank Kutzschebauch for helpful and interesting discussions and critical remarks. 
The authors would like to thank Gerald Schwarz for pointing out Procesi's article \cite{MR419491} on invariant functions of matrices.

\section*{Funding}

The first author was supported by the European Union (ERC Advanced grant HPDR, 101053085 to Franc Forstneri\v{c}) and grant N1-0237 from ARRS, Republic of Slovenia. The first author also received support by a Research Group Linkage Programme from the Humboldt Foundation. Further, the first author was supported by Schweizerisches Nationalfonds Grant [200021-207335] during several short stays at the University of Bern.
The second author was supported by Schweizerisches Nationalfonds Grant [200021-207335].

\end{document}